\documentclass[a4paper,12pt]{article}
\usepackage{a4wide}
\usepackage{amsmath}
\usepackage{amssymb}
\usepackage{amsthm}
\usepackage{latexsym}
\usepackage{graphicx}
\usepackage[english]{babel}
\usepackage{makeidx}
\usepackage{lineno}

\newtheorem{obs} [subsection]{Remark}

\newtheorem{prop}[subsection]{Proposition}

\newtheorem{teor}[subsection]{Theorem}
\newtheorem{lema}[subsection]{Lemma}
\newtheorem{cor} [subsection]{Corollary}
\newcommand{\Zng}{$\mathbb Z^n$-graded $S$-module}

\def\sdepth{\operatorname{sdepth}}
\def\depth{\operatorname{depth}}
\def\supp{\operatorname{supp}}
\def\deg{\operatorname{deg}}

\def\pd{\operatorname{pd}}

\begin{document}
\selectlanguage{english}
\frenchspacing

\numberwithin{equation}{section}

\title{Depth and Stanley depth of powers of the path ideal of a path graph}
\author{Silviu B\u al\u anescu$^1$ and Mircea Cimpoea\c s$^2$}
\date{}

\maketitle

\footnotetext[1]{ \emph{Silviu B\u al\u anescu}, University Politehnica of Bucharest, Faculty of
Applied Sciences, %Department of Mathematical Methods and Models, 
Bucharest, 060042, E-mail: silviu.balanescu@stud.fsa.upb.ro}
\footnotetext[2]{ \emph{Mircea Cimpoea\c s}, University Politehnica of Bucharest, Faculty of
Applied Sciences, %Department of Mathematical Methods and Models, 
Bucharest, 060042, Romania and Simion Stoilow Institute of Mathematics, Research unit 5, P.O.Box 1-764,
Bucharest 014700, Romania, E-mail: mircea.cimpoeas@upb.ro,\;mircea.cimpoeas@imar.ro}

\begin{abstract}
Let $I_{n,m}:=(x_1x_2\cdots x_m,\;  x_2x_3\cdots x_{m+1},\; \ldots,\; x_{n-m+1}\cdots x_n)$ be the
$m$-path ideal of the path graph of length $n$, in the ring $S=K[x_1,\ldots,x_n]$.
We prove that: $$\depth(S/I_{n,m}^t)=\begin{cases} n -t+2 -
 \left\lfloor \frac{n-t+2}{m+1} \right\rfloor - \left\lceil \frac{n-t+2}{m+1} \right\rceil, & t \leq n+1-m
 \\ m-1,& t > n+1-m \end{cases},\text{ for all }t\geq 1.$$ 
Also, we prove that $\depth(S/I_{n,m}) \geq \sdepth(S/I_{n,m}^t) \geq \depth(S/I_{n,m}^t)$
and \linebreak $\sdepth(I_{n,m}^t)\geq \depth(I_{n,m}^t)$, for all $t\geq 1$.


\end{abstract}

\section*{Introduction}
\linenumbers

Let $K$ be a field and $S=K[x_1,\ldots,x_n]$ the polynomial ring over $K$.
% We denote $\me=(x_1,\ldots,x_n)$ the maximal graded ideal of $S$.
Let $M$ be a \Zng. A \emph{Stanley decomposition} of $M$ is a direct sum $\mathcal D: M = \bigoplus_{i=1}^r m_i K[Z_i]$ as a $\mathbb Z^n$-graded $K$-vector space, where $m_i\in M$ is homogeneous with respect to $\mathbb Z^n$-grading, $Z_i\subset\{x_1,\ldots,x_n\}$ such that $m_i K[Z_i] = \{um_i:\; u\in K[Z_i] \}\subset M$ is a free $K[Z_i]$-submodule of $M$. We define $\sdepth(\mathcal D)=\min_{i=1,\ldots,r} |Z_i|$ and $\sdepth(M)=\max\{\sdepth(\mathcal D)|\;\mathcal D$ is a Stanley decomposition of $M\}$. The number $\sdepth(M)$ is called the \emph{Stanley depth} of $M$. 

Herzog, Vladoiu and Zheng show in \cite{hvz} that $\sdepth(M)$ can be computed in a finite number of steps if $M=I/J$, where $J\subset I\subset S$ are monomial ideals. In \cite{rin}, Rinaldo give a computer implementation for this algorithm, in the computer algebra system $\mathtt{CoCoA}$ \cite{cocoa}. We say that a \Zng $\;M$ satisfies the Stanley inequality, if 
$$\sdepth(M)\geq \depth(M).$$
In \cite{apel}, J.\ Apel restated a conjecture firstly given by Stanley in \cite{stan}, namely that any \Zng $\;M$ satisfies the Stanley
inequality. This conjecture proves to be false, in general, for $M=S/I$ and $M=J/I$, where $0\neq I\subset J\subset S$ are monomial ideals, see \cite{duval}.

Stanley depth is an important combinatorial invariant and we believe that it deserves a thorough study. The explicit computation of the Stanley depth it is a difficult task, even in very simple cases, like the maximal monomial ideal $\mathbf m=(x_1,\ldots,x_n)$ of $S$.
Therefore, although the Stanley conjecture was disproved in the most general set up, it is interesting to find large classes of ideals which satisfy the Stanley inequality. 
Also, we note that, in the case of monomial ideals, Stanley's conjecture, i.e. $$\sdepth(I)\geq \depth(I),$$
remains open. For a friendly introduction in the thematic of Stanley depth, we refer the reader \cite{her}.

For $n\geq m\geq 1$, the \emph{$m$-path ideal of the path graph} of length $n$ is:
$$I_{n,m}=(x_1x_2\cdots x_m,\;  x_2x_3\cdots x_{m+1},\; \ldots,\; x_{n-m+1}\cdots x_n)\subset S.$$
In \cite[Theorem 1.3]{path} we proved that: 
$$\sdepth(S/I_{n,m})=\depth(S/I_{n,m})=n+1 - \left\lfloor \frac{n+1}{m+1} \right\rfloor - \left\lceil \frac{n+1}{m+1} \right\rceil.$$
In the preprint \cite{stef}, Alin \c Stefan stated that:
$$\sdepth(S/I_{n,2}^t) = \max\left\{\ \left\lceil \frac{n+t-1}{3} \right\rceil , 1\right\},$$
where $t\geq 1$. %but the proof of the inequality $"\leq"$ is incomplete.

In Theorem \ref{t2}, we generalized both results above, proving that:
$$\sdepth(S/I_{n,m}^t) \geq \depth(S/I_{n,m}^t) = \begin{cases} n -t+2 -
 \left\lfloor \frac{n-t+2}{m+1} \right\rfloor - \left\lceil \frac{n-t+2}{m+1} \right\rceil, & t \leq n+1-m
 \\ m-1,& t > n+1-m \end{cases}.$$
As a consequence, in Corollary \ref{coro}, we give a formula for the projective dimension of $S/I_{n,m}^t$.

Also, in Theorem \ref{t3}, we prove that:
\begin{align*}
& \sdepth(I_{n,m}^t)\geq \depth(I_{n,m}^t)=\begin{cases} n -t+3 -
 \left\lfloor \frac{n-t+2}{m+1} \right\rfloor - \left\lceil \frac{n-t+2}{m+1} \right\rceil, & t \leq n+1-m
 \\ m,& t > n+1-m \end{cases} \\
& \text{ and } \sdepth(I_{n,m}^t)\leq \min\{ n+1 - \left\lfloor \frac{n-t+1}{m+1} \right \rfloor,\; n - \left\lfloor \frac{\lceil \frac{t}{m} \rceil + 1}{2} \right\rfloor \}.
\end{align*}
Hence, $I_{n,m}^t$ satisfies Stanley's inequality, for any $t\geq 1$. 

\newpage
\section{Preliminaries}

First, we recall the well known Depth Lemma, see for instance \cite[Lemma 2.3.9]{real}. % or \cite[Lemma 3.1.4]{vasc}.

\begin{lema}\label{l11}(Depth Lemma)
If $0 \rightarrow U \rightarrow M \rightarrow N \rightarrow 0$ is a short exact sequence of modules over a local ring $S$, or a Noetherian graded ring with $S_0$ local, then:
\begin{enumerate}
\item[(1)] $\depth M \geq \min\{\depth N,\depth U\}$.
\item[(2)] $\depth U \geq \min\{\depth M,\depth N +1\}$.
\item[(3)] $\depth N \geq \min\{\depth U-1,\depth M\}$.
\end{enumerate}
\end{lema}

In \cite{asia}, Asia Rauf proved the analog of Lemma $1.1(1)$ for $\sdepth$:

\begin{lema}\label{asia}
If $0 \rightarrow U \rightarrow M \rightarrow N \rightarrow 0$ be a short exact sequence of $\mathbb Z^n$-graded $S$-modules, then:
$$ \sdepth(M) \geq \min\{\sdepth(U),\sdepth(N) \}.$$
\end{lema}

We also recall the following well known results. See for instance \cite[Corollary 1.3]{asia}, \cite[Proposition 2.7]{mirci},
\cite[Theorem 1.1]{mir}, \cite[Lemma 3.6]{hvz} and \cite[Corollary 3.3]{asia}.

\begin{lema}\label{lem}
Let $I\subset S$ be a monomial ideal and let $u\in S$ a monomial which is not in $I$. Then:
\begin{enumerate}
\item[(1)] $\sdepth(S/(I:u))\geq \sdepth(S/I)$.
\item[(2)] $\sdepth(I:u)\geq \sdepth(I)$.
\item[(3)] $\depth(S/(I:u))\geq \depth(S/I)$.
\item[(4)] If $I=u(I:u)$, then:
  \begin{enumerate}
  \item $\sdepth(S/(I:u))=\sdepth(S/I)$.
	\item $\depth(S/(I:u))=\depth(S/I)$.
	\item $\sdepth(I:u)=\sdepth(I)$.
  \end{enumerate}
\item[(5)] If $u$ is regular on $S/I$, then:
  \begin{enumerate}
	\item $\sdepth(S/(I,u))=\sdepth(S/I)-1$.
  \item $\depth(S/(I,u))=\depth(S/I)-1$.
  \end{enumerate}
\item[(6)] If $S'=S[x_{n+1}]$, then:
\begin{enumerate}
\item $\sdepth_{S'}(S'/IS')=\sdepth_S(S/I)+1$, $\sdepth_{S'}(IS')=\sdepth_S(I)+1$.
\item $\depth_{S'}(S'/IS')=\depth_S(S/I)+1$.
\end{enumerate}
\end{enumerate}
\end{lema}

\begin{teor}(\cite[Theorem 1.3]{path})\label{mirpat}

If $I_{n,m}=(x_1\cdots x_m,\; x_2\cdots x_{m+1},\; \ldots,\; x_{n-m+1}\cdots x_n) \subset S$, then:
$$\sdepth(S/I_{n,m})=\depth(S/I_{n,m})= n+1 - \left\lfloor \frac{n+1}{m+1} \right\rfloor - \left\lceil \frac{n+1}{m+1} \right\rceil.$$
\end{teor}

\section{Main results}

Let $1\leq m\leq n$ be two integers. As in the previous section, we consider the ideal: 
$$I_{n,m}=(x_1\cdots x_m,\; x_2\cdots x_{m+1},\; \ldots,\; x_{n-m+1}\cdots x_n)\subset S.$$

First, we prove the following lemma:

\begin{lema}\label{inmt}
Let $1\leq m\leq n$ and $t\geq 2$ be some integers. Then:
 $$(I_{n,m}^t:x_{n-m+1}\cdots x_n)=I_{n,m}^{t-1}.$$
\end{lema}

\begin{proof}
Since $x_{n-m+1}\cdots x_n\in G(I_{n,m})$, the inclusion "$\supseteq$" is clear. In order to prove the converse inclusion, let 
$u\in (I_{n,m}^t:x_{n-m+1}\cdots x_n)$ be a monomial, i.e. $(x_{n-m+1}\cdots x_n)u \in I_{n,m}^t$. It follows that
there exists $w\in G(I_{n,m}^t)$ such that $w|(x_{n-m+1}\cdots x_n)u$. 

We let $k:=\max\{j\;:\;x_j|w\}$. If $k\leq n-m$ then $w|(x_{n-m+1}\cdots x_n)u$ implies $w|u$, hence $u\in I_{n,m}^t\subseteq I_{n,m}^{t-1}$.
If $k\geq n-m+1$, since $w\in G(I_{n,m}^t)$, then $x_{k-m+1}\cdots x_k|w$. Therefore $w=(x_{k-m+1}\cdots x_k)w'$ such that 
$w'\in G(I_{n,m}^{t-1})$. Moreover, since $w|(x_{n-m+1}\cdots x_n)u$,
it follows that $w'|(x_{k+1}\cdots x_n)u$ and thus $u\in I_{n,m}^{t-1}$, as required. 
\end{proof}

\begin{lema}\label{inmt2}
Let $1\leq m\leq n$, $2\leq k\leq m$ and $t\geq 2$ be some integers. Then:
 $$((I_{n,m}^t:x_{n-k+2}\cdots x_n),x_{n-m+1}\cdots x_{n-k+1} )=(I_{n-k,m}^t,x_{n-m+1}\cdots x_{n-k+1} ).$$
\end{lema}

\begin{proof}
The inclusion "$\supseteq$" is obvious. In order to prove the converse inclusion "$\subseteq$", we choose a monomial
$u\in (I_{n,m}^t:x_{n-k+2}\cdots x_n)$. If $x_{n-m+1}\cdots x_{n-k+1}|u$ then there is nothing to prove.
Assume that $x_{n-m+1}\cdots x_{n-k+1}\nmid u$.

Since $u\in (I_{n,m}^t:x_{n-k+2}\cdots x_n)$, it follows that there exists $w\in G(I_{n,m}^t)$ such that $w|(x_{n-k+2}\cdots x_n)u$.
Note that $w=w_1\cdots w_t$ with $w_i\in G(I_{n,m})$ for all $1\leq i\leq t$. Since $x_{n-m+1}\cdots x_{n-k+1}\nmid u$, we can
deduce that $x_{n-m+1}\cdots x_{n-k+1}\nmid w_i$ for all $1\leq i\leq t$. Hence, $w_i\in G(I_{n-k,m})$ for all $i$ and therefore
$w\in G(I_{n-k,m}^t)$. It follows that $w|u$ and therefore $u\in I_{n-k,m}^t$, as required.
\end{proof}

\begin{prop}\label{t1}
With the above notation, we have:
$$\sdepth(S/I_{n,m}^t) ,\depth(S/I_{n,m}^t) \geq \begin{cases} n -t+2 -
  \left\lfloor \frac{n-t+2}{m+1} \right\rfloor - \left\lceil \frac{n-t+2}{m+1} \right\rceil, & t \leq n+1-m
  \\ m-1,& t > n+1-m \end{cases}.$$
\end{prop}

\begin{proof}
We denote 
$$\varphi(n,m,t):=\begin{cases} n -t+2 -
 \left\lfloor \frac{n-t+2}{m+1} \right\rfloor - \left\lceil \frac{n-t+2}{m+1} \right\rceil, & t \leq n+1-m
 \\ m-1,& t > n+1-m \end{cases}.$$
In order to prove that $\sdepth(S/I_{n,m}^t)\geq \varphi(n,m,t)$, we use induction on $n,m,t\geq 1$. 

If $t=1$, from Theorem \ref{mirpat} it follows that:
$$\sdepth(S/I_{n,m})=n+1 - \left\lfloor \frac{n+1}{m+1} \right\rfloor - \left\lceil \frac{n+1}{m+1} \right\rceil = \varphi(n,m,1),$$
hence we are done. Now, assume $t\geq 2$. %If $m>n$ then there is nothing to prove, since $I_{n,m}=0$.
If $n=m$, then $I_{n,n}=(x_1\cdots x_n)$ and, consequently, $I_{n,n}^t=(x_1^t\cdots x_n^t)$ is a principal ideal. Hence, from 
Lemma \ref{lem}(5), we get:
$$\sdepth(S/I_{n,n}^t)=n-1=\varphi(n,n,t),$$
and there is nothing to prove. 

If $m=1$, then $I_{n,1}^t=(x_1,\ldots,x_n)^t=\mathbf m^t$, where $\mathbf m=(x_1,\ldots,x_n)$ is the graded maximal monomial ideal of $S$. Hence, we have:
$$\sdepth(S/I_{n,1}^t)=\sdepth(S/\mathbf m^t)=0=\varphi(n,1,t),$$
and there is nothing to prove. Thus, we may assume $n>m\geq 2$. 

If $n\leq 2m-1$, then 
\begin{align*}
& I_{n,m}=x_{n-m+1}\cdots x_{m} \widetilde I_{n,m},\text{ where } \\
& \widetilde I_{n,m}=(x_1x_2\cdots x_{n-m},x_2\cdots x_{n-m}x_{m+1},\ldots, x_{n-m}x_{m+1}\cdots x_{n-1},x_{m+1}\cdots x_n).
\end{align*}
It follows that $I_{n,m}^t=x_{n-m+1}^t \cdots x_{m}^t \widetilde I_{n,m}^t$. Therefore, we have that:
\begin{equation}\label{eq1}
\frac{S}{I_{n,m}^t} = \frac{S}{x_{n-m+1}^t \cdots x_{m}^t \widetilde I_{n,m}^t}.
\end{equation}
On the other hand, we have:
\begin{equation}\label{eq2}
\frac{S}{\widetilde I_{n,m}^t} = \left(\frac{K[x_1,\ldots,x_{n-m},x_{m+1},\ldots,x_n]}{\widetilde I_{n,m}^t}\right)[x_{n-m+1},\ldots,x_m].
\end{equation}
Also, by renumbering the variables, we note that:
\begin{equation}\label{eq3}
\frac{K[x_1,\ldots,x_{n-m},x_{m+1},\ldots,x_n]}{\widetilde I_{n,m}^t} \cong \frac{K[x_1,\ldots,x_{2(n-m)}]}{I_{2(n-m),n-m}^t}.
\end{equation}
From \eqref{eq1}, \eqref{eq2}, \eqref{eq3}, Lemma \ref{lem}(4,6) and the induction hypothesis, it follows that:
$$ \sdepth(S/I_{n,m}^t) = \sdepth(S_{2(n-m)}/I_{n-m,m}^t) + 2m - n  \geq  $$
\begin{equation}\label{doin-m}
\geq \varphi(2(n-m),n-m,t)+2m-n = \varphi(n,m,t),
\end{equation}
where we denoted $S_{k}:=K[x_1,\ldots,x_k]$.

In the following, we assume $n\geq 2m$.
We let $L_0:=I_{n,m}^t$ and, inductively, we consider the monomial ideals:
\begin{equation}\label{uj}
 L_j:=(L_{j-1}:x_{n-m+j}),\; U_j:=(L_{j-1},x_{n-m+j})\text{ for }1\leq j\leq m.
\end{equation}
We consider the short exact sequences:
\begin{align*}
& 0 \longrightarrow \frac{S}{L_1} \longrightarrow \frac{S}{L_0} \longrightarrow \frac{S}{U_1}
\longrightarrow 0 \\
& 0 \longrightarrow \frac{S}{L_2} \longrightarrow \frac{S}{L_1} \longrightarrow 
\frac{S}{U_2} \longrightarrow 0  
\end{align*}
$$ \vdots $$
\begin{equation}\label{exacte}
0 \rightarrow \frac{S}{L_m} \rightarrow \frac{S}{L_{m-1}} \rightarrow 
\frac{S}{U_m} \rightarrow 0. 
\end{equation}
From Lemma \ref{inmt} it follows that 
$L_m=(I_{n,m}^t:x_{n-m+1}\cdots x_n)=I_{n,m}^{t-1}$. Hence, by induction hypothesis, it follows that:
\begin{equation}\label{ec2}
\sdepth(S/L_m)=\sdepth(S/(I_{n,m}^t:x_{n-m+1}\cdots x_n)) \geq \varphi(n,m,t-1) \geq \varphi(n,m,t).
\end{equation}
Note that, from \eqref{ec2} and Lemma \ref{lem}(1) it follows that:
\begin{equation}\label{ec3}
\sdepth(S/I_{n,m}^t)\leq \sdepth(S/I_{n,m}^{t-1}).
\end{equation}
Also, from \eqref{uj}, it follows that:
$$ U_j = (L_{j-1},x_{n-m+j}) = (( I_{n,m}^t:x_{n-m+1}\cdots x_{n-m+j-1}), x_{n-m+j}) = $$
$$ = ((I_{n,m}^t,x_{n-m+j}):x_{n-m+1}\cdots x_{n-m+j-1}) = $$
$$= ((I_{n-m+j-1,m}^t,x_{n-m+j}):x_{n-m+1}\cdots x_{n-m+j-1}) =$$
\begin{equation}\label{ujj}
 = ((I_{n-m+j-1,m}^t:x_{n-m+1}\cdots x_{n-m+j-1}),x_{n-m+j}).
\end{equation}
In particular, $U_1=(I_{n,m}^t,x_{n-m+1})=(I_{n-m,m}^t,x_{n-m+1})$, hence, by induction hypothesis and Lemma \ref{lem}, we have:
\begin{equation}\label{u1}
\sdepth(S/U_1)\geq \varphi(n-m,m,t)+m-1.
\end{equation}
From Euclid's division lemma, there exists some integers $q$ and $r$ such that
\begin{equation}\label{euclid}
n-t+2=(m+1)q+r,\text{ where }0\leq r\leq m.
\end{equation}
It follows that:
\begin{equation}\label{pisi1}
\left\lfloor \frac{n-t+2}{m+1} \right\rfloor + \left\lceil \frac{n-t+2}{m+1} \right\rceil =
 2q +  \left\lfloor \frac{r}{m+1} \right\rfloor + \left\lceil \frac{r}{m+1} \right\rceil = \begin{cases} 2q,& r=0 \\ 2q+1,& r\geq 1 \end{cases}.
\end{equation}
Also:
\begin{equation}\label{pisi2}
\left\lfloor \frac{n-m-t+2}{m+1} \right\rfloor + \left\lceil \frac{n-m-t+2}{m+1} \right\rceil =
 2q - \left\lceil \frac{r-m}{m+1} \right\rceil - \left\lfloor \frac{r-m}{m+1} \right\rfloor  = \begin{cases} 2q-1,& r<m \\ 2q,& r = 1 
\end{cases}.
\end{equation}
From \eqref{pisi1} and \eqref{pisi2}, it follows that $\varphi(n-m,m,t)+m-1\geq \varphi(n,m,t)$. Hence, from \eqref{u1},
it follow that:
\begin{equation}\label{uu1}
\sdepth(S/U_1)\geq \varphi(n,m,t).
\end{equation}
Now, we fix $2\leq j\leq m$ and we let $w_j:=x_{n-2m+j}\cdots x_{n-m}$. We let $A_{j,0}:=(U_j,w_j)$ and
$$ A_{j,\ell}:=(A_{j,\ell-1}:x_{n-m-\ell+1})\text{ and }B_{j,\ell}:=(A_{j,\ell-1},x_{n-m-\ell+1}),\text{ for }1\leq \ell \leq m-j.$$
From \eqref{ujj} it follows that:
$$A_{j,\ell}=(A_{j,0}:x_{n-m}\cdots x_{n-m-\ell+1}) = ((U_j:x_{n-m}\cdots x_{n-m-\ell+1}),x_{n-2m+j}\cdots x_{n-m-\ell}) = $$
\begin{equation}\label{ajell}
= ((I_{n-m+j-1,m}^t:x_{n-m-\ell+1}\cdots x_{n-m+j-1}),x_{n-2m+j}\cdots x_{n-m-\ell},x_{n-m+j}),
\end{equation}
for all $0\leq \ell\leq m-j$. Applying Lemma \ref{inmt2} for $n-m+j-1$ instead of $n$ and $k=j+\ell$, it follows that:
\begin{equation}\label{ajl}
A_{j,\ell}= (I_{n-m-\ell-1,m}^t, x_{n-2m+j}\cdots x_{n-m-\ell}  ,x_{n-m+j}),\text{ for all }0\leq \ell\leq m-j.
\end{equation}
We consider the short exact sequences:
\begin{align*}
& 0 \longrightarrow \frac{S}{(U_j:w_j)} \longrightarrow \frac{S}{U_j} \longrightarrow \frac{S}{A_{j,0}}
\longrightarrow 0 \\
& 0 \longrightarrow \frac{S}{A_{j,1}} \longrightarrow \frac{S}{A_{j,0}} \longrightarrow 
\frac{S}{B_{j,1}} \longrightarrow 0  
\end{align*}
$$ \vdots $$
\begin{equation}\label{exacte2}
0 \rightarrow \frac{S}{A_{j,m-j}} \rightarrow \frac{S}{A_{j,m-j-1}} \rightarrow 
\frac{S}{B_{j,m-j}} \rightarrow 0. 
\end{equation}
From \eqref{ujj} and Lemma \ref{inmt}, it follows that:
\begin{equation}\label{ujw}
(U_j:w_j)=((I_{n-m+j-1,m}^t:x_{n-2m+j}\cdots x_{n-m+j-1}),x_{n-m+j})=(I_{n-m+j-1,m}^{t-1},x_{n-m+j})
\end{equation}
Hence, by induction hypothesis, it follows that:
\begin{equation}\label{ujpw}
\sdepth(S/(U_j:w_j)) \geq \varphi(n-m+j-1,m,t-1)+m-j.
\end{equation}
By straightforward computations we note that: 
\begin{equation}\label{cutu}
\varphi(n-m+j-1,m,t-1)+m-j \geq \varphi(n-1,m,t-1) = \varphi(n,m,t),\text{ for all }1\leq j\leq m.
\end{equation}
Therefore, from \eqref{ujpw} and \eqref{cutu}, we obtain:
\begin{equation}\label{ujws}
\sdepth(S/(U_j:w_j)) \geq \varphi(n,m,t).
\end{equation}
% From \eqref{uj}, it follows that:
% \begin{equation}
% (U_j,w_j)=((I_{n-m+j-1,m}^t:x_{n-m+1}\cdots x_{n-m+j-1}),w_j,x_{n-m+j}).
% \end{equation}
% On the other hand, we have that $((I_{n-m+j-1,m}^t:x_{n-m+1}\cdots x_{n-m+j-1}),w_j)=(I_{n-m-1,m}^t,w_j)$, therefore
% \begin{equation}\label{aj0}
% A_{j,0}=(U_j,w_j) = (I_{n-m-1,m}^t,w_j,x_{n-m+j}).
% \end{equation}
% Using similar arguing, it follows that
From %\eqref{aj0} and 
\eqref{ajl}, it follows that
\begin{equation}\label{bjl}
B_{j,\ell}=(A_{j,\ell-1},x_{n-m-\ell+1})= (I_{n-m-\ell,m}^t,  x_{n-m-\ell+1}  ,x_{n-m+j}),\text{ for all }1\leq \ell\leq m-j.
\end{equation}
From \eqref{ajl}, Lemma \ref{lem}, the induction hypothesis and straightforward computations, it follows that:
\begin{equation*}
\sdepth(S/A_{j,m-j})=\sdepth(S/(I_{n-2m+j-1,m}^t, x_{n-2m+j}, x_{n-m+j})) \geq 
\end{equation*}
\begin{equation}\label{ajms}
\geq \varphi(n-2m+j-1,m,t)+2m-j-1 \geq \varphi(n,m,t).
\end{equation}
Similarly, from \eqref{bjl} and Lemma \ref{lem} it follows that
\begin{equation}\label{bjls}
\sdepth(S/B_{j,\ell}) = \varphi(n-m-\ell,m,t)+m+\ell \geq \varphi(n,m,t),\text{ for all }1\leq \ell\leq m-j.
\end{equation}
Now, from \eqref{ujws}, \eqref{ajms}, \eqref{bjls} and the short exact sequences \eqref{exacte2}, we conclude that
\begin{equation}\label{ujura}
\sdepth(S/U_j)\geq \varphi(n,m,t),\text{ for all }2\leq j\leq m.
\end{equation}
Also, from \eqref{ec2}, \eqref{uu1}, \eqref{ujura} and the short exact sequences \eqref{exacte}, we conclude that
\begin{equation}\label{maimare}
\sdepth(S/I_{n,m}^t)\geq \varphi(n,m,t).
\end{equation}
The proof of the inequality $\depth(S/I_{n,m}^t)\geq \varphi(n,m,t)$ is similar, using Lemma \ref{l11}(2) instead of Lemma \ref{asia}
and, also, the statements in Lemma \ref{lem} regarding $\depth$.
\end{proof}

Let $t,m\geq 2$ be two integers. %Assume that $t+m=ma+b$, where $1\leq b\leq m$.
In the ring $S_{t+m}:=K[x_1,x_2,\ldots,x_{t+m}]$, we consider the monomial ideal:
$$U_{m,t}=(x_{i_1}\cdots x_{i_m} \in S_{m+t}\;:\; i_j\cong j (\bmod\; m),\;1\leq j\leq m ).$$
% \{\widehat{i_1},\ldots,\widehat{i_m}\}=\{\widehat{0},\ldots,\widehat{m-1}\}),$$
% where $\widehat{a}$ is the class of $a$ in $\mathbb Z/m\mathbb Z$. 
\begin{lema}\label{l1}
With the above notations, we have that: 
\begin{enumerate}
\item[(1)] $\depth(S_{t+m}/U_{m,t})=m-1$.
\item[(2)] $m-1\leq \sdepth(S_{t+m}/U_{m,t})\leq t+m-1-\left\lceil \frac{t}{m} \right\rceil$.
\item[(3)] $\sdepth(U_{m,t}) \leq  t+m -\left\lfloor \frac{\lceil \frac{t}{m} \rceil + 1}{2} \right\rfloor$.
\item[(4)] $\sdepth(U_{m,t}) \geq  t+m - (t+m - m\lceil \frac{t}{m} \rceil) \left\lfloor \frac{\lceil \frac{t}{m} \rceil + 1}{2} \right\rfloor - 
           (m\lceil \frac{t}{m} \rceil - t)\left\lfloor \frac{\lceil \frac{t}{m} \rceil}{2} \right\rfloor$.
\end{enumerate}
\end{lema}

\begin{proof}
 (1) Assume that $t+m=ma+b$, where $1\leq b\leq m$. Then:
 \begin{equation}\label{umtv}
 U_{m,t}=V_{m,1,a+1}\cap \ldots \cap V_{m,b,a+1}\cap V_{m,b+1,a}\cap \ldots \cap V_{m,m,a},\text{ where }V_{m,j,k}:=
   (x_j,x_{j+m},\ldots,x_{j+(k-1)m}).
 \end{equation}
 Note that, we have a partition: 
 $$G(V_{m,1,a+1})\cup \cdots \cup G(V_{m,b,a+1}) \cup G(V_{m,b+1,a})\cup \cdots \cup G(V_{m,m,a}) = \{x_1,x_2,\ldots,x_{t+m}\}.$$
 Therefore, since $V_{m,j,k}$ is the maximal monomial ideal of $K[G(V_{m,j,k})]$, from the definition of $U_{m,t}$ it follows, using an     
 inductive argument, that $\depth(S/U_{m,t})=m-1$.

 (2) According to \cite[Corollary 1.9(3)]{mirci}, since $U_{m,t}$ is the intersection of $m$ ideals in disjoint sets of variables,
 it follows that $\sdepth(S/U_{m,t})\geq m-1$. On the other hand, according to \cite[Theorem 1.3(2)]{mirci}, we have that:
 $$\sdepth(S/U_{m,t})\leq \sdepth(S/V_{m,1,a+1})=t+m-a-1 = t+m-1-\left\lceil \frac{t}{m} \right\rceil.$$
 (3) According to \cite[Theorem 1.1]{ishaq} and \eqref{umtv}, we have that: 
 $$\sdepth(U_{m,t})\leq t+m - \left\lfloor \frac{a+1}{2} \right\rfloor = t+m -\left\lfloor \frac{\lceil \frac{t}{m} \rceil + 1}{2} \right\rfloor.$$
 (4) According to \cite[Corollary 1.8]{ishaq} and \eqref{umtv}, we have that:
 \begin{align*}
   & \sdepth(U_{m,t})\geq t+m - b\left\lfloor \frac{a+1}{2} \right\rfloor - (m-b)\left\lfloor \frac{a}{2} \right\rfloor = \\
   & = t+m - (t+m - m\left\lceil \frac{t}{m} \right\rceil) \left\lfloor \frac{\lceil \frac{t}{m} \rceil + 1}{2} \right\rfloor - 
	(m\left\lceil \frac{t}{m} \right\rceil - t)\left\lfloor \frac{\lceil \frac{t}{m} \rceil}{2} \right\rfloor.
 \end{align*} 
 Hence, we get the required result.
\end{proof}

Let $q\geq 1$, $t,m\geq 2$, $0\leq r\leq m$ and $n:=(m+1)q+t-1+r$. We consider the monomials:
\begin{align*}
& w(m,t):=(x_2\cdots x_{m+1})(x_3\cdots x_{m+2})\cdots (x_t\cdots x_{t+m-1})\text{ and }\\
& w(m,t,q):=w(m,t)\cdot v(m,t,q),
\text{ where }v(m,t,q)=\prod_{\ell=1}^{q-1}x_{t+\ell(m+1)+1}\cdots x_{t+\ell(m+1)+m-1}.
\end{align*}

As usual, given a monomial $v\in S$, the \emph{support} of $v$, denoted by $\supp(v)$,
is the set of variables which divide $v$.

\begin{lema}\label{l2}
With the above notations, we have that:
\begin{align*}
& (I_{n,m}^t:w(m,t,q)) = \begin{cases} U_{m,t}+P_{m,t,q},& r<m \\ U_{m,t}+P_{m,t,q}+(x_{n-m+1}\cdots x_n),& r=m \end{cases}, \text{ where } \\
& P_{m,t,q}:=(x_{t+m+1},x_{t+2(m+1)},\ldots,x_{t+(q-1)(m+1)})+(x_{t+2m+1},x_{t+3m+2},\ldots,x_{t+q(m+1)-1}).
\end{align*} \small{
Also, $P_{m,t,q}=(\{x_{t+m+1},\ldots,x_{n-r}\}\setminus \supp(v(m,t,q)))$ and $|\supp(v(m,t,q))|=(m-1)(q-1)+r$.}
\end{lema}

\begin{proof}
In order to prove the equality, we use double inclusion.

First, note that $w(m,t)$ is a minimal monomial generator of $I_{n,m}^{t-1}$.

If $q\geq 2$ and $1\leq \ell\leq q-1$, then 
\begin{equation}\label{ref1}
x_{t+\ell(m+1)}x_{t+\ell(m+1)+1}\cdots x_{t+\ell(m+1)+m-1} \;|\; x_{t+\ell(m+1)}v(m,t,q) \text{ and }
\end{equation}
\begin{equation}\label{ref2}
 x_{t+\ell(m+1)+1}\cdots x_{t+\ell(m+1)+m-1}x_{t+\ell(m+1)+m} \;|\; x_{t+\ell(m+1)+m}v(m,t,q).
\end{equation}
As $x_{t+\ell(m+1)}x_{t+\ell(m+1)+1}\cdots x_{t+\ell(m+1)+m-1}$, 
$x_{t+\ell(m+1)+1}\cdots x_{t+\ell(m+1)+m-1}x_{t+\ell(m+1)+m} \in G(I_{n,m})$, 
$w(m,t)\in G(I_{n,m}^{t-1})$ and $w(m,t,q)=w(m,t)v(m,t,q)$, from \eqref{ref1} and \eqref{ref2} it follows that:
$$x_{t+\ell(m+1)}w(m,t,q),\; x_{t+\ell(m+1)+m}w(m,t,q) \in I_{n,m}^{t}.$$
Hence, we obtain:
\begin{equation}\label{pmtq}
P_{m,t,q} \subset (I_{n,m}^t:w(m,t,q)).
\end{equation}
If $r=m$, since $x_{n+m-1}\cdots x_n\in G(I_{n,m})$, $w(m,t)\in G(I_{n,m}^{t-1})$ and $w(m,t)|w(m,t,q)$, then:
\begin{equation}\label{regalm}
x_{n+m-1}\cdots x_n\in (I_{n,m}^t:w(m,t,q)).
\end{equation}
Given a proper monomial $u\in S$, we denote 
$$\max(u)=\max\{i\;:\;x_i|u\},\;\min(u)=\min\{i\;:\;x_i|u\}\text{ and }||u||=\max(u)-\min(u)+1.$$
We choose $u\in G(U_{m,t})$, i.e. $u=x_{i_1}\cdots x_{i_m}$ with $i_j\cong j(\bmod\;m)$ and we note that $||u||\geq m$.
If $||u||=m$, then $u$ is the product of $m$ consecutive monomials, hence $u\in G(I_{n,m})$.
Therefore $u\cdot w(m,t,q)\in I_{n,m}^t$. Now, assume $||u||>m$. We have:
\begin{align*}
& u\cdot w(m,t) = (x_2\cdots x_{m+1}) \cdots (x_{\min(u)-1}\cdots x_{\min(u)+m-2}) (x_{\min(u)}\cdots x_{\min(u)+m-1}) \cdot \\
& (x_{\min(u)+2}\cdots x_{\min(u)+m+1})\cdots (x_{t}\cdots x_{t+m-1})\cdot x_{\min(u)+m} u/ x_{\min(u)}.
\end{align*}
Note that $w'(m,t):= w(m,t)x_{\min(u)}/x_{\min(u)+m} \in G(I_{n,m}^{t-1})$.
We let $u':=x_{\min(u)+m} u/ x_{min(u)}$.

It is easy to see that $||u'||<||u||$ and $u\cdot w(m,t)=u'\cdot w'(m,t)$.

If $||u'||=m$ then $u'\in G(I_{n,m})$ and,
from above, it follows that $u\cdot w(m,t)\in G(I_{n,m}^t)$.

If $||u'||>m$ then we repeat the same procedure and we obtain $u'':=x_{\min(u')+m} u'/x_{\min(u')}$, with $||u''||<||u'||$.
Since $\min(u')>\min(u)$, we can write:
$$u\cdot w(m,t) = u' \cdot w'(m,t) = u'' \cdot w''(m,t),$$
with $w''(m,t)\in G(I_{n,m}^{t-1})$. If $||u''||=m$, then we are done. Otherwise, we repeat the same procedure until we can
find some $\ell\geq 2$ such that $u\cdot w(m,t) = u^{(\ell)}\cdot w^{(\ell)}(m,t)$, with $w^{(\ell)}(m,t)\in G(I_{n,m}^{t-1})$
and $|u^{(\ell)}||=m$.

Finally, we get $u\cdot w(m,t,q)\in I_{n,m}^t$ and we obtain:
\begin{equation}\label{utm}
U_{t,m} \subset (I_{n,m}^t:w(m,t,q)).
\end{equation}
From \eqref{pmtq}, \eqref{regalm} and \eqref{utm} we complete the proof of "$\subseteq$".
In order to prove the other inclusion, let $u\in S$ be a monomial such that $u\cdot w(m,t,q)\in I_{n,m}^t$.
We write $u=u'\cdot u''$ with $u'\in S_{t+m}$ and $u''\in K[x_{t+m+1},\ldots,x_n]$.

If $u\notin P_{m,t,q}$ then $\supp(u) \cap \{x_{t+m+1}, x_{t+m+2}, \ldots, x_{n-r}\} \subseteq \supp(v(m,t,q))$.
Also, if $r=m$ and $u\notin (x_{n-m+1}\cdots x_n)$, then $\{x_{n-m+1},\ldots, x_n\} \nsubseteq \supp(u)$.

Since $x_{t+m+1}\notin\supp(u)$, $u''\cdot v(t,m,q) \notin I_{n,m}^t$ and $u \cdot w(m,t,q)\in I_{n,m}^t$,
it follows that $u'w(t,m)\in I_{n,m}^t$. 

Given a monomial $v\in S_{t+m}$, $v=x_{a_1}x_{a_2}\cdots x_{a_d}$, where $d=\deg(v)$, we let 
$$\deg_{\ell}(v) = |\{a_i\;:\;a_i\equiv \ell(\bmod\;m)\}|,\text{ for }1\leq \ell\leq m.$$
It is easy to see that $u\in I_{t+m,m}^t$ if and only if $\deg_{\ell}(v)\geq t$ for all $1\leq \ell\leq m$.
Indeed, any minimal monomial generator of $I_{m+t,m}$ is the product of $m$ consecutive variables.

We assume, by contradiction, that $u'\notin U_{m,t}$. It follows that
there exists $1\leq k\leq m$ such that $\deg_{k}(u')=0$. On the other hand, $\deg_{\ell}(w(t,m))=t-1$,
for all $1\leq \ell\leq m$. It follows that $\deg_k(u'w(t,m))=t-1<t$, hence $u'w(t,m)\notin I_{n,m}^t$,
a contradiction.
\end{proof}

\begin{teor}\label{t2}
With the above notations, we have that:
\begin{enumerate}
\item[(1)] $\sdepth(S/I_{n,m}^t) \geq \depth(S/I_{n,m}^t) = \begin{cases} n -t+2 -
 \left\lfloor \frac{n-t+2}{m+1} \right\rfloor - \left\lceil \frac{n-t+2}{m+1} \right\rceil, & t \leq n+1-m
 \\ m-1,& t > n+1-m \end{cases}$.
\item[(2)] $\sdepth(S/I_{n,m}^t) \leq \sdepth(S/I_{n,m}) = n+1 - \left\lfloor \frac{n+1}{m+1} \right\rfloor - \left\lceil \frac{n+1}{m+1} \right\rceil$.
\end{enumerate}
\end{teor}

\begin{proof}
 (1) We denote:
$$\varphi(n,m,t):=\begin{cases} n -t+2 -
 \left\lfloor \frac{n-t+2}{m+1} \right\rfloor - \left\lceil \frac{n-t+2}{m+1} \right\rceil, & t \leq n+1-m
 \\ m-1,& t > n+1-m \end{cases}.$$
From Proposition \ref{t1}, it follows that $\depth(S/I_{n,m}^t)\geq \varphi(n,m,t)$.
Therefore, in order to prove that $\depth(S/I_{n,m}^t)=\varphi(n,m,t)$ it suffice to show
that $\depth(S/I_{n,m}^t)\leq \varphi(n,m,t)$. In order to do this, we use induction on $n,m$ and $t$.

The cases $t=1$, $n=m$, $m=1$ and $n\leq 2m-1$ are proved easily, as in the proof of Proposition \ref{t1}.
Now, assume $n\geq 2m$ and $t\geq 2$. If $n\leq m+t-1$, then, by induction hypothesis, Lemma \ref{lem}(3) and Lemma \ref{inmt}, 
it follows that:
\begin{align*}
& \depth(S/I_{n,m}^t) \leq \depth(S/(I_{n,m}^t:x_{n-m+1}\ldots x_n)) = \depth(S/I_{n,m}^{t-1}) \leq \\
& \leq \varphi(n,m,t-1) = m-1 = \varphi(n,m,t).
\end{align*}
Now, assume $n\geq m+t-1$. By Euclid's division, it follows that there exists $q\geq 1$ and $0\leq r\leq m$, such that:
\begin{equation}\label{cute}
 n = q(m+1) +t - 1 + r.
\end{equation}
Asumme $r<m$. According to Lemma \ref{l2}, we have that:
$$ S/(I_{n,m}^t:w(n,m,t)) = (K[x_1,\ldots,x_{m+t}]/U_{m,t})\otimes_K (K[x_{m+t+1},\ldots,x_n]/P_{m,t,q}).$$
Hence, from  Lemma \ref{l1}(1) and Lemma \ref{lem}, it follows that:
$$ \depth(S/I_{n,m}^t) \leq \depth(S/(I_{n,m}^t:w(n,m,t)))=(m-1)+(m-1)(q-1)+r=(m-1)q+r.$$
On the other hand, from \eqref{cute}, we have that:
$$\varphi(n,m,t) = n-t+2 -\left\lfloor \frac{n-t+2}{m+1} \right\rfloor - \left\lceil \frac{n-t+2}{m+1} \right\rceil = 
(m+1)q +1 +r - q - (q+1) =(m-1)q +r,$$
hence $\depth(S/I_{n,m}^t)\leq \varphi(n,m,t)$.
Similarly, in the case $r=m$, since $x_{n-m+1}\cdots x_n$ is regular in $K[x_{m+t+1},\ldots,x_n]/P_{m,t,q}$,
by Lemma \ref{lem}(5), we get:
$$\depth(S/I_{n,m}^t)\leq \depth(S/(I_{n,m}^t:w(n,m,t))) = (m-1)q+r-1 = \varphi(n,m,t).$$
Therefore, $\depth(S/I_{n,m}^t)\leq \varphi(n,m,t)$, for any $n,m$ and $t$.

(2) Since $(I_{n,m}^t:(x_{n-m+1}\cdots x_n)^{t-1})=I_{n,m}$, the required result follows from Lemma \ref{lem}(1) and
Theorem \ref{mirpat}. See also \eqref{ec3}.
\end{proof}

\begin{obs}\label{obs1}\rm
From Lemma \ref{l1}(2), using a similar technique as in the proof of Theorem \ref{t2}, we can deduce that
$\sdepth(S/I_{n,m}^t) \leq  \varphi(n,m,t)+t-\left\lceil \frac{t}{m} \right\rceil.$

Note that $\varphi(n,m,1)\leq \varphi(n,m,t)+t-\lceil \frac{t}{m} \rceil$, hence this upper bound does not improve
the one given in Theorem \ref{t2}(2).
\end{obs}

\begin{cor}\label{coro}
The projective dimension of $S/I_{n,m}^t$ is:
$$\pd(S/I_{n,m}^t) = \begin{cases} t-2 +
 \left\lfloor \frac{n-t+2}{m+1} \right\rfloor + \left\lceil \frac{n-t+2}{m+1} \right\rceil, & t \leq n+1-m
 \\ n-m+1,& t > n+1-m \end{cases}.$$
\end{cor}

\begin{proof}
It follows immediately from Theorem \ref{t2} and Ausl\"ander-Buchsbaum's Theorem; 
see \cite[Theorem 3.5.13]{real}.
\end{proof}

\begin{teor}\label{t3}
With the above notation, we have:
\begin{enumerate} 
\item[(1)] $\sdepth(I_{n,m}^t)\geq \depth(I_{n,m}^t)=\begin{cases} n -t+3 -
 \left\lfloor \frac{n-t+2}{m+1} \right\rfloor - \left\lceil \frac{n-t+2}{m+1} \right\rceil, & t \leq n+1-m
 \\ m,& t > n+1-m \end{cases}$.
\item[(2)] $\sdepth(I_{n,m}^t)\leq \min\{ n+1 - \lfloor \frac{n-t+1}{m+1} \rfloor,\; n - \left\lfloor \frac{\lceil \frac{t}{m} \rceil + 1}{2} \right\rfloor \}$.
\end{enumerate}
\end{teor}

\begin{proof}
(1) From Theorem \ref{t2} we have that $\depth(I_{n,m}^t)=\varphi(n,m,t)+1$, where 
$$\varphi(n,m,t) = \begin{cases} 
n -t+2 - \left\lfloor \frac{n-t+2}{m+1} \right\rfloor - \left\lceil \frac{n-t+2}{m+1} \right\rceil, & t \leq n+1-m
 \\ m, & t > n+1-m \end{cases},$$
hence, in order to prove the first assertion of the theorem, we have to show
that: $$\sdepth(I_{n,m}^t)\geq \varphi(n,m,t)+1 \text{ or }\sdepth(I_{n,m}^t)\geq\depth(I_{n,m}^t).$$
We use induction on $n,m,t\geq 1$.

If $t=1$ then $\sdepth(I_{n,m})\geq \depth(I_{n,m})$ follows from \cite[Proposition 1.7]{lucrare}. 

Now, assume $t\geq 2$.
If $n=m$ then $I_{n,n}=(x_1\cdots x_n)$ and $I_{n,n}^t=(x_1^t\cdots x_n^t)$ is a principal ideal. Hence:
$$\sdepth(I_{n,n}^t)=n=\depth(I_{n,n}^t).$$
Also, if $m=1$ then $I_{n,1}=\mathbf m=(x_1,\ldots,x_n)$ and we obviously have that:
$$\sdepth(I_{n,1}^t)=\sdepth(\mathbf m^t)\geq 1 = \depth(\mathbf m^t).$$
Thus, we may assume that $n>m\geq 2$. 
If $n\leq 2m-1$, then 
\begin{align*}
& I_{n,m}=x_{n-m+1}\cdots x_{m} \widetilde I_{n,m},\text{ where } \\
& \widetilde I_{n,m}=(x_1x_2\cdots x_{n-m},x_2\cdots x_{n-m}x_{m+1},\ldots, x_{n-m}x_{m+1}\cdots x_{n-1},x_{m+1}\cdots x_n).
\end{align*}
It follows that $I_{n,m}^t=x_{n-m+1}^t \cdots x_{m}^t \widetilde I_{n,m}^t$. As in the proof of Proposition \ref{t1}, we
have that: $$I_{n,m}^t \cong I_{2(n-m),n-m}^t[x_{n-m+1},\ldots,x_n],$$
and, therefore, by induction hypothesis and Lemma \ref{lem}, it follows that:
$$\sdepth(I_{n,m}^t)\geq \varphi(2(n-m),n-m,t)+1+2m-n=\varphi(n,m,t)+1,$$
as required. In the following, we assume $n\geq 2m$.

We let $L_0:=I_{n,m}^t$ and $L_j:=(L_{j-1}:x_{n-m+j})$ for $1\leq j\leq m$. We have
the decompositions:
$$ I_{n,m}^t = L_0 = x_{n-m+1}L_1 \oplus L_0/x_{n-m+1}L_1 $$
$$ L_1 = x_{n-m+2}L_2 \oplus L_1/x_{n-m+2}L_2 $$
$$ \vdots $$
\begin{equation}\label{lemee}
L_{m-1}=x_{n}L_m \oplus L_{m-1}/x_nL_m.
\end{equation}
From Lemma \ref{inmt}, we have that $L_m=(I_{n,m}^t:x_{n-m+1}\cdots x_n)=I_{n,m}^{t-1}$ and hence, by
induction hypothesis and Lemma \ref{lem}, we obtain:
\begin{equation}\label{lemeee}
\sdepth(x_nL_m)=\sdepth(L_m)=\sdepth(I_{n,m}^{t-1})\geq \varphi(n,m,t-1)+1 \geq \varphi(n,m,t)+1.
\end{equation}
According to \eqref{lemeee} and the decompositions \eqref{lemee}, in order to prove that
$$\sdepth(I_{n,m}^t)\geq \varphi(n,m,t)+1,$$ it is enough to show that
\begin{equation}\label{evrika}
\sdepth(L_{j-1}/x_{n-m+j}L_j)\geq \varphi(n,m,t)+1,\text{ for all }1\leq j\leq m.
\end{equation}
Using the identity $v(I:v)=(v)\cap I$, where $v\in S$ is a monomial and $I\subset S$ is a monomial ideal, it follows that that:
$$ \frac{L_{j-1}}{x_{n-m+j}L_j} = \frac{(I_{n,m}^t:x_{n-m+1}\cdots x_{n-m+j-1})}{x_{n-m+j}(I_{n,m}^t:x_{n-m+1}\cdots x_{n-m+j})}
\cong $$ 
\begin{equation}\label{proclet}
 \cong \frac{x_{n-m+1}\cdots x_{n-m+j-1} (I_{n,m}^t:x_{n-m+1}\cdots x_{n-m+j-1})}{x_{n-m+1}\cdots x_{n-m+j}(I_{n,m}^t:x_{n-m+1}\cdots x_{n-m+j})}
\cong \frac{I_{n,m}^t\cap (x_{n-m+1}\cdots x_{n-m+j-1})}{I_{n,m}^t\cap (x_{n-m+1}\cdots x_{n-m+j})}.
\end{equation}
We claim that:
\begin{equation}\label{claimu}
\frac{I_{n,m}^t\cap (x_{n-m+1}\cdots x_{n-m+j-1})}{I_{n,m}^t\cap (x_{n-m+1}\cdots x_{n-m+j})} \cong 
(x_{n-2m+j}\cdots x_{n-m+j-1})I_{n-m+j-1,m}^{t-1}[x_{n-m+j+1},\ldots,x_n].
\end{equation}
Indeed, if $u\in G(I_{n,m}^t)$ such that $x_{n-m+1}\cdots x_{n-m+j-1}|u$ and $x_{n-m+1}\cdots x_{n-m+j}\nmid u$, then 
it is easy to note that $u \in G(I_{n-m+j-1,m}^t)$ and $x_{n-2m+j}\cdots x_{n-m+j-1}|u$. We write
$$u=(x_{n-2m+j}\cdots x_{n-m+j-1})w,\;\text{ where }w\in K[x_1,\ldots,x_{n-m+j-1}].$$
From Lemma \ref{inmt}, it follows that:
$$w\in (I_{n-m+j-1,m}^t:x_{n-2m+j}\cdots x_{n-m+j-1}) = I_{n-m+j-1,m}^{t-1}.$$
In order to complete the proof of the claim \eqref{claimu}, it is enough to notice that 
$$u\cdot u' \in (I_{n,m}^t \cap (x_{n-m+1}\cdots x_{n-m+j-1})) \setminus (I_{n,m}^t \cap (x_{n-m+1}\cdots x_{n-m+j})),$$
if and only if $u'$ is a monomial with $x_{n-m+j}\notin \supp(u')$.

Now, from \eqref{proclet}, \eqref{claimu}, Lemma \ref{lem} and the induction hypothesis, it follows that:
$$\sdepth(L_{j-1}/x_{n-m+j}L_j) = \sdepth(I_{n-m+j-1,m}^{t-1})+m-j \geq \varphi(n-m+j-1,m,t-1) + 1+m-j \geq $$
$$\geq n-m+j-1-(t-1)+2 - \left\lfloor \frac{n-m+j-t+2}{m+1} \right\rfloor - \left\lceil \frac{n-m+j-t+2}{m+1} \right\rceil + 1+ m-j =$$
$$
 = n-t+2 - \left\lfloor \frac{n-m+j-t+2}{m+1} \right\rfloor - \left\lceil \frac{n-m+j-t+2}{m+1} \right\rceil + 1 \geq
\varphi(n,m,t)+1,$$
hence, \eqref{evrika} holds. Thus, the first part of the proof is complete.

(2) According to Lemma \ref{lem} and \cite[Theorem 1.3]{mirci}, we have that:
   $$\sdepth(I_{n,m}^t)\leq \sdepth(I_{n,m}^t:w(m,t,q))\leq \min\{ \sdepth(U_{m,t}S),\sdepth(P_{m,t,q}S)\},$$
where $q=\lfloor \frac{n-t+1}{m+1} \rfloor$. The conclusion follows from the fact that $P_{n,t,q}$ has $2(q-1)$
generators, if $q\geq 2$, and Lemma \ref{l1}(3).

\end{proof}

\subsection*{Aknowledgments} 

We gratefully acknowledge the use of the computer algebra system Cocoa (\cite{cocoa}) for our experiments.

The second author was supported by a grant of the Ministry of Research, Innovation and Digitization, CNCS - UEFISCDI, 
project number PN-III-P1-1.1-TE-2021-1633, within PNCDI III.

\end{document}